\documentclass[oneside,reqno,12pt]{amsart}

\usepackage{amsmath,amssymb,amsfonts}
\usepackage{kpfonts}
\usepackage{mathtools}
\usepackage{bbm}
\usepackage{hyperref}
\usepackage[foot]{amsaddr}

\usepackage{bbm}

\newtheorem{thm}{Theorem}
\newtheorem{lemma}{Lemma}

\theoremstyle{definition}
\newtheorem{defin}{Definition}

\theoremstyle{remark}
\newtheorem{example}{Example}
\newtheorem{assumption}{Assumption A\hspace*{-4pt}}
\newtheorem{remark}{Remark}

\newcommand{\ii}{\ensuremath{\mathbf{1}}}
\newcommand{\rr}{\ensuremath{\mathbb{R}}}
\newcommand{\zz}{\ensuremath{\mathbb{Z}}}
\newcommand{\nn}{\ensuremath{\mathbb{N}}}
\newcommand{\cc}{\ensuremath{\mathbb{C}}}
\newcommand{\xx}{{\sf{X}}}
\newcommand{\pp}{{\sf{P}}}
\newcommand{\mm}{{\sf{m}}}
\newcommand{\llf}{{\sf{L}}}
\newcommand{\bb}{\mathcal{B}}
\newcommand{\ff}{\mathcal{F}}
\newcommand{\pr}{{\partial}}

\begin{document}

\title[Multidimensional integral with respect to stochastic measure]{Sample path properties of multidimensional integral with respect to stochastic measure}

\author{Boris Manikin$^1$}

\author{Vadym Radchenko$^1$}

\address
{$^1$Department of Mathematical Analysis, Taras Shevchenko National University of Kyiv, 01601 Kyiv, Ukraine. Emails:bmanikin@knu.ua, vadymradchenko@knu.ua}

\begin{abstract}
The integral with respect to a multidimensional stochastic measure, for which we assume only $\sigma$-additivity in probability, is studied. The continuity and differentiability of its realizations are established.
\end{abstract}


\maketitle

\section{Introduction}\label{scintr}

The main purpose of this article is studying the regularity property of the integral of a deterministic function with respect to stochastic measure $\mu$, where $\mu$ is defined on Borel subsets of $[0,1]^d$. In basic statements of the paper, for $\mu$ we assume only $\sigma$-additivity in probability. We study the regularity with respect to a parameter and with respect to the upper limit of the integral. Case $d=1$ was considered in~\cite[Section 2.3]{radbook}, but methods of \cite{radbook} do not work for $d\ge 2$. In our paper, we study the integral using the Fourier -- Haar expansion of the integrand.

Sample path behaviour of various stochastic processes was studied in many works; we mention only some of them. The conditions for sample boundedness and continuity of $\alpha$-stable processes were established in~\cite[Chapter 10]{samtaq}. H{\"o}lder continuity of harmonizable operator scaling stable random field is given by Corollary 5.5~\cite{bier09}. In Theorem 4.1 from~\cite{pani21} the uniform modulus of continuity for some self-similar $S\alpha S$ random field was obtained, while the upper bound of modulus of continuity of stable random fields is given by Proposition~5.1 and Corollary~5.3~\cite{bier15}. These results, for example, imply H{\"o}lder continuity of mentioned fields.

Conditions of continuity of Gaussian random fields may be found in Theorems 3.3.3, 3.4.1 and 3.4.3~\cite{adler81}. H{\"o}lder continuity of centered Gaussian random fields under certain conditions is proved in~\cite[Theorem 4.2]{xiao09}. Theorem~2.1~\cite{xiao10} gives an estimate of modulus of continuity of general random fields on metric space, these results are applied to harmonizable stable random fields.

In our article, we consider stochastic processes which are integrals with respect to a general stochastic measure. The definition of these measures, their properties and examples as well as information about Fourier -- Haar series may be found in Section \ref{sc:prel}. The continuity and differentiability of sample paths of parameter-dependent integral are established in Section \ref{sc:param}. The continuity of realizations of the integral as a function of the upper limit is studied in Section \ref{sc:upper}.

\section{Preliminaries}\label{sc:prel}

\subsection{Stochastic measures}\label{ss:stochm}

In this subsection, we give basic information concerning stochastic measures in a general setting. In statements of Sections~\ref{sc:param} and \ref{sc:upper}, this set function is defined on Borel subsets of $[0,1]^d$.

Let $\llf_0=\llf_0(\Omega, {\ff}, {\pp} )$ be the set of all real-valued
random variables defined on the complete probability space $(\Omega, {\ff}, {\pp} )$. Convergence in $\llf_0$ means the convergence in probability. Let ${\xx}$ be an arbitrary set and ${\bb}$ a $\sigma$-algebra of subsets of ${\xx}$.

\begin{defin}\label{def:stme}
A $\sigma$-additive mapping $\mu:\ {\bb}\to \llf_0$ is called {\em stochastic measure} (SM).
\end{defin}
In other words,
\begin{eqnarray}
\mu(A\cup B)=\mu(A)+\mu(B) \text{ for } A\cap B=\emptyset,\nonumber\\
\mu(A_n)\stackrel{\pp}{\to}0 \text{ for } A_n\downarrow \emptyset.
\label{condmu}
\end{eqnarray}

We do not assume the moment existence or martingale properties for SM. We can say that $\mu$ is an $\llf_0$--valued measure.

We note the following examples of SMs in multidimensional spaces. All measures, if nothing else is mentioned, are considered on Borel $\sigma$-algebra of sets in $[0,1]^d$.

\begin{example}
Let $\mu$ be a random orthogonal measure --- it is defined, for example, in \cite[Section 5.3]{gikhsko96} and \cite[Section 2.3]{krylov02} --- with finite structural function. Then $\mu$ is an SM in the sense of Definition \ref{def:stme}; condition \eqref{condmu} follows from \cite[Section 5.3, Theorem 1(d)]{gikhsko96}. For example, the set function $\mu(A)=\int_{[0,1]^d}{\ii}_A(t)\,{d}W(t)$, where we take multiple Wiener-It\^{o} integral (see its properties, for example, in \cite[Section 1.1]{nual06}), is a random orthogonal measure with structural function $\mm(A)=d!\lambda_d(A)$, where $\lambda_d$ is the $d$-dimensional Lebesgue measure.
\end{example}

\begin{example}
Let $\mu$ be an independently scattered $\alpha$-stable random measures for $\alpha\in (0,2]$ with finite control measure. According to the definition  --- see \cite[Definition 3.3.1]{samtaq} --- for disjoint Borel sets $\{A_n:n\geq 1\}$ the following holds:
\[
\mu(\cup_{n=1}^\infty A_n)=\sum_{n=1}^\infty \mu(A_n)\ \text{a. s.}
\]
Therefore, $\mu$ is an SM.
\end{example}

\begin{example}
\label{exhermite}
Consider the set function $\mu:\mathcal{B}([0,1]^d)\to \llf_0$ such that
\begin{equation}
\label{eqherm}
\mu(A)=\int_{\mathbb{R}^d} {\ii}_A(t)\,{d}Z_{H}^q(t),
\end{equation}
where $Z_{H}^q$ is the Hermite sheet (see, for example, \cite{tudor14} or \cite[Section 4]{tudor23}), $H=(H_1,\dots,H_d)$, $1/2<H_i<1$. Under these conditions, integral in \eqref{eqherm} is well defined, and $\mu$ is an SM in the sense of Definition \ref{def:stme}; the statement \eqref{condmu} can be proved using the representation of $\mu$ via Wiener integral (see \cite[(4.12)]{tudor23}). \cite[Proposition 3]{tudor14} implies the H\"{o}lder continuity of $\mu$.

If $q=1$ then the Hermite sheet coincides with the fractional Brownian sheet.
\end{example}

Theorem~10.1.1 of~\cite{kwawoy} states the sufficient conditions under which the product of one-dimensional SMs generates an SM.

SMs may be used for study of stochastic dynamical systems (see~\cite{bai22}).

For deterministic measurable functions $f:\xx\to\rr$, an integral of the form $\int_{\xx}f\,d\mu$ is studied
in~\cite[Chapter 7]{kwawoy}, ~\cite[Chapter 1]{radbook}. In particular, every bounded measurable $f$ is integrable w.~r.~t. any~$\mu$. An analogue of the Lebesgue dominated convergence theorem holds for this integral (see~\cite[Proposition 7.1.1]{kwawoy} or~\cite[Theorem 1.5]{radbook}).

Below we will use the following statement for SMs defined on arbitrary $\sigma$-algebra $\bb$.

\begin{lemma} \label{lmfkmu} (Lemma 3.1\cite{radt09}) Let $\mu$ be an SM on $(\xx,\bb)$, $f_k:\ {\xx}\to \rr,\ k\ge 1$, be measurable functions such that
$ \hat{f}(x)=\sum_{k=1}^{\infty} |{f_k}(x)|$ is integrable w.~r.~t.~$\mu$. Then
\begin{equation*}
\sum_{k=1}^{\infty}\Bigl(\int_{\xx} f_k\,d\mu \Bigr)^2<\infty\quad
\mbox{\textrm ~a.~s.}
\end{equation*}
\end{lemma}

\subsection{Besov spaces}

We recall the definition of Besov spaces $B^\alpha_{p,p}([0, 1]^d)$ following~\cite{devore} and~\cite{kamont}. Other approaches to the definition and properties of these spaces are considered in~\cite{beilni}, \cite[Sections 2.1 and 2.4]{runst}, \cite{triebel}, \cite[Sections 2 and 5]{tristr}. The equivalence of different definitions is discussed in~\cite{devore}, \cite[Section 17.2]{leoni}, \cite[Section 2.3]{runst}.

For functions $f\in \llf_p([0, 1]^d)$, we consider the value
\begin{equation*}
\|f\|_{B^\alpha_{p,p}([0, 1]^d)}=\|f\|_{\llf_{p}([0, 1]^d)}+\Bigl(\int_0^{1} (\omega_p(f, r))^p r^{-\alpha p-1}
\,dr\Bigr)^{1/p},
\end{equation*}
where $\omega_p$ denotes the $\llf_p$-modulus of continuity,
\begin{eqnarray*}
\omega_p(f, r) & = & \sup_{|h|\le r}\Bigl(\int_{I_{h}} |f(x+h)-f(x)|^p\,dx\Bigr)^{1/p},\\
{\textrm{where}}\quad I_{h} & = & \{x\in [0, 1]^d:\ x+h\in [0, 1]^d \},
\end{eqnarray*}
and $|h|$ denotes the Euclidean norm in $\rr^d$. Then
\[
B^\alpha_{p,p}([0, 1]^d)=\{f\in \llf_p([0, 1]^d):\ \|f\|_{B^\alpha_{p,p}([0, 1]^d)}<+\infty\},
\]
and $\|\cdot\|_{B^\alpha_{p,p}([0, 1]^d)}$ is a norm in this space.

Let $\mu$ be an SM defined on the Borel $\sigma$-algebra of~$[0, 1]^d$. For
$x=(x_1, x_2, \dots, x_d)\in [0, 1]^d$ set
\[
\mu(x)=\mu\Bigl(\prod_{i=1}^{d}[0, x_i]\Bigr).
\]

In the sequel, we will refer to the following assumption.

\begin{assumption}\label{assbm} There exists a real-valued finite measure~$\mm$
on~$(\xx, \bb)$ with the following property: if a measurable function $g:\xx\to\rr$ satisfies
\mbox{$\int_{\xx}g^2\,d\mm<+\infty$} then $g$ is integrable w.~r.~t.~$\mu$ on~$\xx$.
\end{assumption}

For orthogonal measure, we can take its structural measure as $\mm$.  For an $\alpha$-stable random measure A\ref{assbm} holds for the control measure $\mm$ (see (3.4.1)~\cite{samtaq}). For an SM from Example \ref{exhermite} we can take measure $\mm$ such that
$$
\mm(A)=\int_Adv\int_{[0,1]^d}\prod_{i=1}^d|u_i-v_i|^{2H_i-2}du,
$$
where $u$ and $v$ are $d$-dimensional vectors with the components $u_i$, $v_i$.

We have the following statement concerning the Besov regularity of SMs defined on Borel subsets of $[0,1]^d$.

\begin{thm}\label{thnbspd} (Theorem 5.1~\cite{radt09}) Let the random function
\[
\mu(x)=\mu\Bigl(\prod_{i=1}^{d}[0,x_i]\Bigr),\ x\in[0,1]^d,
\]
has continuous paths and Assumption~A\ref{assbm} holds.

Then for any $1\le p<+\infty,\ 0<\alpha<\min\{1/p,1/2\}$ the realization $\mu(x)$, $x\in[0,1]^d$ with probability 1 belongs to the Besov space~$B^\alpha_{p,p}([0,1]^d)$.
\end{thm}

\subsection{Haar functions}\label{ss:haar}

In order to construct a version of the stochastic integral, we use the approximation of the integrand with Fourier--Haar series.

We follow the definition of one-dimensional Haar functions $\chi_n(x)$, $x\in [0,1]$, from ~\cite[Section 3]{kassaa}. For $n=2^j+i$, $1\le i \le 2^{j}$, $j\ge 0$ we write
\begin{equation}\label{eq:defdel}
\begin{split}
\Delta_n=\Delta_j^i=\bigl((i-1)2^{-j},i2^{-j}\bigr),\quad \overline{\Delta}_n=\bigl[(i-1)2^{-j},i2^{-j}\bigr],\\
\Delta_1=\Delta_0^0=(0,1),\quad \overline{\Delta}_1=[0,1],\\
\Delta_n^+=\bigl(\Delta_j^i\bigr)^+=\bigl((i-1)2^{-j},(2i-1)2^{-j-1}\bigr)=\Delta_{j+1}^{2i-1},\\
\Delta_n^-=\bigl(\Delta_j^i\bigr)^-=\bigl((2i-1)2^{-j-1},i2^{-j}\bigr)=\Delta_{j+1}^{2i}.
\end{split}
\end{equation}

Now let $\chi_1=1$, and
\begin{equation*}
\chi_n(x)= \left\{
\begin{array}{l}
0,\quad x\notin \overline{\Delta}_n,\\
2^{j/2},\quad x\in \Delta_n^+,\\
-2^{j/2},\quad x\in \Delta_n^-.
\end{array}
\right.
\end{equation*}
for  $2^j+1\le n \le 2^{j+1}$. In discontinuity points and at the endpoints of $[0,1]$ we define
\begin{equation}
\label{discp}
\begin{split}
\chi_n(x)=\bigl(\chi_n(x-)+\chi_n(x+)\bigr)/2,\quad x\in(0,1),\\
\chi_n(0)=\chi_n(0+),\quad \chi_n(1)=\chi_n(1-).
\end{split}
\end{equation}

For $g\in \llf_1([0,1])$, we define the Fourier -- Haar coefficients and sums in a standard way:
\begin{eqnarray*}
c_n(g)=\int_{[0,1]}g(x)\chi_n(x)\,dx=2^{j/2}\Bigl(\int_{\Delta_n^+}g(x)\,dx-\int_{\Delta_n^-}g(x)\,dx\Bigr),\\
S_N(g,x)=\sum_{n=1}^N c_n(g) \chi_n(x).
\end{eqnarray*}
We will approximate $g$ taking the Fourier -- Haar sums for values $N=2^k$, $k\in\zz_+$.

By (III.8), (III.8$'$)~\cite{kassaa}, for $x\ne i2^{-k}$, $0\le i\le 2^k$, holds
\begin{equation}\label{eq:stwokx}
S_{2^k}(g,x)=\sum_{i=1}^{2^k}2^k\ii_{\Delta_k^i}(x)\int_{\Delta_k^i}g(t)\,dt,
\end{equation}
and
\begin{equation}\label{eq:stwoki}
\begin{split}
S_{2^k}\Bigl(g,\frac{i}{2^k}\Bigr)=2^{k-1}\Bigl(\int_{\Delta_k^i}g(t)\,dt+\int_{\Delta_k^{i+1}}g(t)\,dt\Bigr),\quad i=1,2,\dots 2^{k}-1,\\
S_{2^k}(g,0)=2^k \int_{\Delta_k^1}g(t)\,dt,\quad S_{2^k}(g,1)=2^k \int_{\Delta_k^{2^k}}g(t)\,dt.
\end{split}
\end{equation}

For functions $f:[0,1]^d\to\rr$, we will use approximation by multivariate Fourier -- Haar sums.

The multivariate Haar functions are defined with equalities
\[
\chi^{(d)}_{n_1, \dots, n_d}(x)=\chi_{n_1}(x_1)\chi_{n_2}(x_2)\dots\chi_{n_d}(x_d),
\]
where $\chi_{n_s}$ are one-dimensional Haar functions, $x=(x_1,\dots,x_d)\in [0,1]^d$. Consider Fourier -- Haar coefficients
\[
c^{(d)}_{n_1, \dots, n_d}(f)=\int_{[0,1]^d}f(x)\chi^{(d)}_{n_1, \dots, n_d}(x)\,dx,
\]
and sums
\[
S^{(d)}_{2^k}(f,x)=\sum_{n_1, \dots, n_d\in \{1,2,3,\dots 2^k\}} c^{(d)}_{n_1, \dots, n_d}(f)\chi^{(d)}_{n_1, \dots, n_d}(x).
\]
In the sequel, we denote the set of elements $(n_1,\dots,n_d)$ such that $1\le n_s\le 2^k$ for  all $s, 1\le s\le d$ by $\nn_k$. In other words, $\nn_k=\{1,2,3,\dots, 2^k-1, 2^k\}^d$.

We need the following statement on the uniform convergence of multivariate Fourier -- Haar sums for continuous functions.
For $d=1$ this fact may be found, for example, in \cite[Theorem III.2]{kassaa}. A similar statement for multivariate periodic functions was proved in~\cite{andskop16}.

\begin{lemma}\label{lm:haarun}
If $f\in\cc([0,1]^d)$ then
\[
\sup_{x\in [0,1]^d} |S^{(d)}_{2^k}(f,x)-f(x)|\to 0,\quad k\to\infty.
\]
\end{lemma}

\begin{proof}
For bounded functions $u:[0,1]^d\to\rr$ we use the following notation of the uniform modulus of continuity
\[
\omega(u,r)=\sup_{x, x+h \in [0,1]^d,\ |h|\le r}|u(x+h)-u(x)|.
\]
From \eqref{eq:stwokx} and \eqref{eq:stwoki} (see also (III.15)\cite{kassaa}) it follows that for $d=1$ holds
\begin{equation}\label{eq:sumdif}
|S_{2^k}(g,x)-g(x)|\le \omega(g,2^{-k}).
\end{equation}

For $t=(t_1,\dots t_d)$ we get
\begin{eqnarray*}
S^{(d)}_{2^k}(f,x)=\sum_{1\le n_1\le 2^k}\dots\sum_{1\le n_d\le 2^k}\chi_{n_1}(x_1)\dots\chi_{n_d}(x_d)\\
\times \int_{[0,1]^d}f(t)\chi_{n_1}(t_1)\chi_{n_2}(t_2)\dots\chi_{n_d}(t_d) \,dt.
\end{eqnarray*}

We consider the operators
\[
P_{k,j}(u)=S_{2^k}(u,x_j),
\]
where we take one-dimensional Fourier -- Haar sums of the function $u(x_1,\dots,\ x_d)$ in coordinate $x_j$. Then
\[
S^{(d)}_{2^k}(f,x)=P_{k,d}(P_{k,d-1}(\dots P_{k,1}(f)\dots)).
\]
Estimate \eqref{eq:sumdif} implies
\begin{equation}\label{eq:pfomega}
|P_{k,j}(u)(x)-u(x)|\le \omega(u,2^{-k}),
\end{equation}
and it is easy to see that for integral $m\ge 1$
\begin{equation}\label{eq:puomega}
\omega(P_{k,j}(u),m2^{-k})\le \omega(u,(m+2)2^{-k}).
\end{equation}
We have
\begin{eqnarray*}
|S^{(d)}_{2^k}(f,x)-f(x)|\le |P_{k,1}(f)(x)-f(x)|\\
+ \sum_{l=1}^{d-1}|P_{k,l+1}(P_{k,l}(\dots P_{k,1}(f)\dots))-P_{k,l}(P_{k,l-1}(\dots P_{k,1}(f)\dots))|,
\end{eqnarray*}
where
\begin{eqnarray*}
|P_{k,l+1}(P_{k,l}(\dots P_{k,1}(f)\dots))-P_{k,l}(P_{k,l-1}(\dots P_{k,1}(f)\dots))|\\
\stackrel{\eqref{eq:pfomega}}{\le} \omega(P_{k,l}(\dots P_{k,1}(f)\dots),2^{-k})\stackrel{\eqref{eq:puomega}}{\le}  \omega(f,(2l+1)2^{-k}).
\end{eqnarray*}
Therefore,
\begin{equation}\label{eq:estsf}
|S^{(d)}_{2^k}(f,x)-f(x)|\le \sum_{l=0}^{d-1}\omega(f,(2l+1)2^{-k}).
\end{equation}
For continuous $f$ the right hand side of \eqref{eq:estsf} tends to 0 as $k\to\infty$, that implies the statement of our lemma.
\end{proof}

\begin{remark}
We obtained estimate \eqref{eq:estsf} for all bounded $f\in\llf_1([0,1]^d)$, not necessarily continuous.
\end{remark}

\section{Parameter dependent integral}\label{sc:param}

In this section, we study the properties of random function
\begin{equation}\label{eq:defeta}
\eta(z)=\int_{[0,1]^d}f(z, x)\,d\mu(x),\ z\in Z,
\end{equation}

Note that parameter stochastic integral w.r.t. Brownian sheet was considered in~\cite{sot_tud_2006}, \cite{sot_viit_2015}.

If $f$ is continuous in $x$ then Lemma~\ref{lm:haarun} and analogue of the dominated convergence theorem (\cite[Theorem 1.5]{radbook}) imply that for each fixed $z$
\[
\int_{[0,1]^d}f(z, x)\,d\mu(x)={\rm p}\lim_{k\to\infty} \int_{[0,1]^d}S^{(d)}_{2^k}(f, x)\,d\mu(x).
\]

Therefore,
\begin{equation}\label{eq:defetas}
\tilde{\eta}(z)=\int_{[0,1]^d}S^{(d)}_{1}(f, x)\,d\mu(x)+\sum_{k=1}^{\infty}\int_{[0,1]^d}(S^{(d)}_{2^k}(f, x)-S^{(d)}_{2^{k-1}}(f, x))\,d\mu(x).
\end{equation}
is the version of $\eta(z)$.

In \eqref{eq:defetas} we have integrals of simple functions, and each integral is equal to respective linear combination of values of $\mu$. We fix the same version of $\mu$ for all these values.

We will use notation $e_i=(0,\dots,0,1,0,\dots,0)$, where 1 stays on $i$-th place, $e_i\in\rr^d$.

\begin{lemma}\label{lm1}
Let the function $f(z,x): Z\times [0,1]^d\to\rr$ be continuously differentiable $l$ times on $[0,1]^d$ for each $z$, $l\ge 0$, and
\begin{equation}\label{eqf1}
\Bigl|\dfrac{\pr^{r} f(z,x)}{\pr x_{s_1}\dots\pr x_{s_r}}\Bigr|\le C_f,\ x\in[0,1]^d
\end{equation}
for some constant $C_f>0$, which is independent of~$z$, and all $0\le r\le l$, $s_1,\dots,s_r\subset \{1,\dots,d\}$. Moreover, let $l$-th derivatives be H\"{o}lder continuous with an exponent $\alpha>0$:
\begin{equation}
\label{eqf2}
\Bigl|\dfrac{\partial^{l} f(z,x^{(1)})}{\partial x_{s_1}\dots\partial x_{s_l}}-\dfrac{\partial^{l} f(z,x^{(2)})}{\partial x_{s_1}\dots\partial x_{s_l}}\Bigr|\le C_f|x^{(1)}-x^{(2)}|^\alpha.
\end{equation}

If $l+\alpha>d/2$, then the version~\eqref{eq:defetas} satisfies the inequality
\begin{equation*}
|\tilde{\eta}(z)|\le C_f C^{(d)}_{\mu}(\omega),
\end{equation*}
where $C^{(d)}_{\mu}(\omega)$ denotes a random constant that depends only of dimension $d$ and SM $\mu$, $C^{(d)}_{\mu}(\omega)<\infty$ a.s.

Moreover, series~\eqref{eq:defetas} converge absolutely and uniformly, i. e.
\begin{equation}\label{eq:absuni}
\sum_{k=1}^{\infty}\sup_{z\in Z}\Bigl|\int_{[0,1]^d}(S^{(d)}_{2^k}(f, x)-S^{(d)}_{2^{k-1}}(f, x))\,d\mu(x)\Bigr|\le C_f C^{(d)}_{\mu}(\omega).
\end{equation}

\end{lemma}

\begin{proof}

We have that
\begin{equation}\label{eq:stwok}
\int_{[0,1]^d}S^{(d)}_{2^k}(f, x)\,d\mu(x)=\sum_{(n_1,\dots,n_d)\in\nn_k} c^{(d)}_{n_1, \dots, n_d}(f)\int_{[0,1]^d} \chi^{(d)}_{n_1, \dots, n_d}(x)\,d\mu(x).
\end{equation}

For each $n_s\ge 2$, take $j_s$ such that $2^{j_s}+1\le n_s\le 2^{j_s+1}$, $j_s\ge 0$. Denote by $\{j^{(k)},1\leq k\leq d\}$ the permutation of $\{j_k,1\leq k\leq d\}$ such that
\[
j^{(1)}\leq j^{(2)}\leq \dots \leq j^{(d)}.
\]
If $n_{s_1}=\dots=n_{s_i}=1$, we set $j^{(1)}=\dots=j^{(i)}=0$. We rewrite the expression for $c^{(d)}_{n_1, \dots, n_d}(f)$ in the following way:
\begin{equation}\label{eq:estcn}
\begin{split}
c^{(d)}_{n_1, \dots, n_d}(f)=\int_{[0,1]^d}f(z,t)\chi^{(d)}_{n_1, \dots, n_d}(t)\,dt =\int_{[0,1]^d}f(z,t)\prod_{1\le s\le d} \chi_{n_s}(t_s)\,dt\\
=2^{\sum_{1\le s\le d,n_s\ge 2} j_s/2}\int_{[0,1]^d}f(z,t)\prod_{1\le s\le d,n_s\ge 2}\Bigl(\ii_{\Delta_{n_s}^+}(t_s)-\ii_{\Delta_{n_s}^-}(t_s)\Bigr)\,dt\\
=2^{\sum_{1\le s\le d,n_s\ge 2} j_s/2}\int_{[0,1]^d}f(z,t)\\
\times\prod_{1\le s\le d,n_s\ge 2}\Bigl(\ii_{\Delta_{n_s}^+}(t_s)-\ii_{\Delta_{n_s}^+}(t_s-2^{-j_s-1})\Bigr)\,dt\\
\stackrel{(*)}{=}2^{\sum_{1\le s\le d,n_s\ge 2} j_s/2}\\
\times \int_{[0,1]^d}\sum_{\varepsilon_s\in\{0,1\},n_s\ge 2}(-1)^{\varepsilon_1 +\dots+\varepsilon_d}f(z,t_1+\varepsilon_1 2^{-j_1-1},\dots, t_d+\varepsilon_d 2^{-j_d-1})\\
\times\prod_{1\le s\le d,n_s\ge 2}\ii_{\Delta_{n_s}^+}(t_s)\,dt
=2^{\sum_{1\le s\le d,n_s\ge 2} j_s/2}\\
\times\int_{\prod_{1\le s\le d,n_s\ge 2}\Delta_{n_s}^+}\sum_{\varepsilon_s\in\{0,1\},n_s\ge 2}(-1)^{\varepsilon_1 +\dots+\varepsilon_d}\\
\times f(z,t_1+\varepsilon_1 2^{-j_1-1},\dots, t_d+\varepsilon_d 2^{-j_d-1})\,dt.
\end{split}
\end{equation}
In equality (*) we opened the brackets in the product
\[
\prod_{1\le s\le d,n_s\ge 2}\Bigl(\ii_{\Delta_{n_s}^+}(t_s)-\ii_{\Delta_{n_s}^+}(t_s-2^{-j_s-1})\Bigr)
\]
and used the change of the variables $t_s-2^{-j_s-1}\to t_s$ in $\ii_{\Delta_{n_s}^+}(t_s-2^{-j_s-1})$.

The last sum in \eqref{eq:estcn} can be represented as the sum of at most $2^{d-1}$ summands of the form
\[
f(z,t^*)-f(z,t^*+2^{-j^{(d)}-1}e_{s_d})=\int_0^{2^{-j^{(d)}-1}}\frac{\partial f(z,t^*+h_{s_d})}{\partial t_{s_d}}dh_{s_d},\qquad j^{(d)}=j_{s_d}.
\]
Now we repeat the same thoughts $l-1$ times and use the properties of a function $f$, which leads to the following inequality
\begin{eqnarray}
\label{estc}
|c^{(d)}_{n_1, \dots, n_d}(f)|\leq C_f2^{d-l-1}2^{\sum_{1\le s\le d,n_s\ge 2} -j_s/2}2^{-(j^{(d)}+\dots+j^{(d-l+1)}+\alpha j^{(d-l)})}.
\end{eqnarray}

Further, we have the estimate
\begin{equation*}
\begin{split}
\sum_{k=1}^{\infty}\sup_{z\in Z}\Bigl|\int_{[0,1]^d}(S^{(d)}_{2^k}(f, x)-S^{(d)}_{2^{k-1}}(f, x))\,d\mu(x)\Bigr|\\
\le \sum_{k=1}^{\infty}\sup_{z\in Z} \Bigl|\sum_{(n_1,\dots,n_d)\in\mathbb{N}_k\setminus\mathbb{N}_{k-1}} c^{(d)}_{n_1, \dots, n_d}(f)\int_{[0,1]^d} \chi^{(d)}_{n_1, \dots, n_d}(x)\,d\mu(x)\Bigr|\\
\stackrel{\eqref{estc}}{\leq} 2^{d} C_f  \sum_{k=1}^{\infty}\sum_{(n_1,\dots,n_d)\in\mathbb{N}_k\setminus\mathbb{N}_{k-1}}\hspace{-24pt} 2^{\sum_{1\le s\le d,n_s\ge 2} -j_s/2}2^{-(j^{(d)}+\dots+j^{(d-l+1)}+\alpha j^{(d-l)})}\\
\times\Bigl|\int_{[0,1]^d} \chi^{(d)}_{n_1, \dots, n_d}(x)\,d\mu(x)\Bigr|\\
\leq 2^{d} C_f \biggl( \sum_{k=1}^{\infty}\sum_{(n_1,\dots,n_d)\in\mathbb{N}_k\setminus\mathbb{N}_{k-1}}2^{\beta\sum_{1\le s\le d,n_s\ge 2}j_s}2^{-2(j^{(d)}+\dots+j^{(d-l+1)}+\alpha j^{(d-l)})}\biggr)^{1/2}\\
\times \biggl( \sum_{k=1}^{\infty}\sum_{(n_1,\dots,n_d)\in\mathbb{N}_k\setminus\mathbb{N}_{k-1}}2^{(-\beta-1)\sum_{1\le s\le d,n_s\ge 2}j_s}\Bigl(\int_{[0,1]^d} \chi^{(d)}_{n_1, \dots, n_d}(x)\,d\mu(x)\Bigr)^2\biggr)^{1/2}\\
:=2^{d} C_fP_1^{1/2}P_2^{1/2}.
\end{split}
\end{equation*}
For a sufficiently small $\beta>0$ we have the following estimates for $P_1$:
\begin{eqnarray}
\nonumber
P_1=\sup_{k\geq 1}\sum_{(n_1,\dots,n_d)\in\mathbb{N}_k}2^{\beta\sum_{1\le s\le d,n_s\ge 2}j_s}2^{-2(j^{(d)}+\dots+j^{(d-l+1)}+\alpha j^{(d-l)})}\\
\nonumber
=\sup_{k\geq 1}\sum_{A\subset \{1,\dots,d\}}\sum_{\substack{0\leq j_i\leq k,\\i\in A}}2^{(\beta+1)\sum_{i\in A}j_i-2(j^{(d)}+\dots+j^{(d-l+1)}+\alpha j^{(d-l)})}\\
\leq 2^d\sup_{k\geq 1}\sum_{0\leq j_i\leq k}2^{(\beta+1-2(l+\alpha)/d)\sum_{i=1}^dj_i}=2^dC.
\label{eqP1}
\end{eqnarray}
Here we denoted the set of indexes $i$, for which $n_s=1$, as $A$. It is left to prove that for each fixed $\beta>0$ $P_2\leq C_\mu^{(d)}(\omega)$. As follows from Lemma \ref{lmfkmu}, it is sufficient to show that
\begin{equation}
\tilde{P}_2=\sum_{k=1}^{\infty}\sum_{(n_1,\dots,n_d)\in\mathbb{N}_k\setminus\mathbb{N}_{k-1}}2^{(-\beta/2-1/2)\sum_{1\le s\le d,n_s\ge 2}j_s}| \chi^{(d)}_{n_1, \dots, n_d}(x)|\leq C^{(d)},
\label{eqP2}
\end{equation}
where by $C^{(d)}$ we denote positive constant that depends only on $d$. The inequality \eqref{eqP2} holds true, as is proved below.
\begin{eqnarray*}
\tilde{P}_2=\sup_{k\geq 1}\sum_{(n_1,\dots,n_d)\in\mathbb{N}_k}2^{(-\beta/2-1/2)\sum_{1\le s\le d,n_s\ge 2}j_s}|\chi_{n_1}(x_1)\dots\chi_{n_d}(x_d)|\\
\leq \sup_{k\geq 1}\sum_{(n_1,\dots,n_d)\in\mathbb{N}_k}2^{(-\beta/2)\sum_{1\le s\le d,n_s\ge 2}j_s}\\
\times\prod_{s=1}^d\Bigl(\mathbf{1}_{\Delta_{n_s}}(x_s)+\frac{1}{2}\mathbf{1}_{ \{(i-1)2^{-j_s}\} }(x_s)+\frac{1}{2}\mathbf{1}_{ \{i2^{-j_s}\} }(x_s)\Bigr)\\
\stackrel{(*)}{\leq} \sup_{k\geq 1}\sum_{A\subset \{1,\dots,d\}}\sum_{\substack{0\leq j_s\leq k,\\s\in A}}2^{(-\beta/2)\sum_{s\in A}j_s}\leq 2^d\sup_{k\geq 1}\sum_{0\leq j_s\leq k}2^{(-\beta/2)\sum_{s=1}^dj_s}= C^{(d)}.
\end{eqnarray*}
Here in inequality (*) we used that for each set $j_1,\dots,j_d$, for each $x$, $x_s\ne i2^{-j_s}$,  exists exactly one set $(n_1,\dots,n_d)$ such that $\chi^{(d)}_{n_1, \dots, n^d}(x)\ne 0$ and $2^{j_s}+1\le n_s\le 2^{j_s+1}$. If we have coordinates $x_s= i2^{-j_s}$, we take into account that $\ii_{ \{i2^{-j_s}\} }(x_s)$ has the coefficient $1/2$, as follows from \eqref{discp}. Now the statement of the lemma is a consequence of \eqref{eqP1} and \eqref{eqP2}.
\end{proof}
\begin{remark}
\label{rm1}
Assume that constants in inequalities \eqref{eqf1} and \eqref{eqf2} depend not only on $f$, but also on $z$. Then the series \eqref{eq:defetas} converges a. s. for each fixed $z\in Z$ and the following analogue of \eqref{eq:absuni}:
\[
\sum_{k=1}^{\infty}\Bigl|\int_{[0,1]^d}(S^{(d)}_{2^k}(f, x)-S^{(d)}_{2^{k-1}}(f, x))\,d\mu(x)\Bigr|\le C_{f,z} C^{(d)}_{\mu}(\omega).
\]
This statement is proved similarly to Lemma \ref{lm1}. We just refer to $C_{f,z}$ everywhere instead of $C_f$.
\end{remark}

The following statement gives the conditions of the continuity with respect to parameter $z\in Z$.

\begin{thm}\label{th:conp}
Let $Z$ be a metric space, and the function $f(z,x): Z\times [0,1]^d\to\rr$ be such that all paths $f(\cdot,x)$ are continuous on $Z$ for each $x$, $f(z,\cdot)$ is continuously differentiable $l$ times on $[0,1]^d$ for each $z$, $l\ge 0$, and
\begin{equation*}
\Bigl|\dfrac{\pr^{r} f(z,x)}{\pr x_{s_1}\dots\pr x_{s_r}}\Bigr|\le C_f,\ x\in[0,1]^d
\end{equation*}
for some constant $C_f>0$ (independent of~$z$) and all $0\le r\le l$, $s_1,\dots,s_r\subset \{1,\dots,d\}$. Let $l$-th derivatives be H\"{o}lder continuous with the exponent $\alpha$ and
\begin{equation*}
\Bigl|\dfrac{\partial^{l} f(z,x^{(1)})}{\partial x_{s_1}\dots\partial x_{s_l}}-\dfrac{\partial^{l} f(z,x^{(2)})}{\partial x_{s_1}\dots\partial x_{s_l}}\Bigr|\le C_f|x^{(1)}-x^{(2)}|^\alpha.
\end{equation*}

If, in addition, $l+\alpha>d/2$, then the random function $\eta$ defined by \eqref{eq:defeta} has a version~\eqref{eq:defetas} with continuous paths on $Z$ a.~s.
\end{thm}

\begin{proof}

From \eqref{eq:estcn} it follows that each $c^{(d)}_{n_1, \dots, n_d}$ is continuous function of variable $z$,
and \eqref{eq:stwok} implies that $\int_{[0,1]^d}S^{(d)}_{2^k}(f, x)\,d\mu(x)$ is continuous for all $k\ge 1$, $\omega\in\Omega$.

By \eqref{eq:absuni}, $\int_{[0,1]^d}S^{(d)}_{2^k}(f, x)\,d\mu(x)$ converges to $\tilde{\eta}(z)$ uniformly a.s. as $k\to\infty$, that implies the statement of our theorem
\end{proof}

\begin{thm}
Let the function $f(z,x): Z\times [0,1]^d\to\mathbb{R}$, $Z=[a,b]$, be continuous differentiable $l+1$ times on $Z\times[0,1]^d$ while
\begin{equation*}
\Bigl|\dfrac{\partial^{r+1} f(z,x)}{\partial{z}\partial x_{s_1}\dots\partial x_{s_r}}\Bigr|\le C_f,\ x\in[0,1]^d,
\end{equation*}
for all $0\le r\le l$, $s_1,\dots,s_r\subset \{1,\dots,d\}$. Moreover, let $l+1$-th derivatives be H\"{o}lder continuous with the exponent $\alpha>0$:
\begin{equation*}
\Bigl|\dfrac{\partial^{l+1} f(z,x^{(1)})}{\partial z\partial x_{s_1}\dots\partial x_{s_l}}-\dfrac{\partial^{l+1} f(z,x^{(2)})}{\partial z\partial x_{s_1}\dots\partial x_{s_l}}\Bigr|\le C_f|x^{(1)}-x^{(2)}|^\alpha.
\end{equation*}
If, in addition, $l+\alpha>d/2$, then paths of a random function $\tilde{\eta}$, which is defined by \eqref{eq:defetas}, have bounded derivatives on $Z$:
\[
\dfrac{d\tilde{\eta}(z)}{dz}=\int_{[0,1]^d}\dfrac{\partial f(z, x)}{\partial z}\,d\mu(x).
\]
\end{thm}

\begin{proof}
Lemma \ref{lm1} implies that
\begin{eqnarray*}
\int_{[0,1]^d}\dfrac{\partial f(z, x)}{\partial z}\,d\mu(x)=\int_{[0,1]^d}S^{(d)}_{1}\Bigl(\frac{\partial f}{\partial z}, x\Bigr)\,d\mu(x)\\
+\sum_{k=1}^{\infty}\int_{[0,1]^d}\Bigl(S^{(d)}_{2^k}\Bigl(\frac{\partial f}{\partial z}, x\Bigr)-S^{(d)}_{2^{k-1}}\Bigl(\frac{\partial f}{\partial z}, x\Bigr)\Bigr)\,d\mu(x),
\end{eqnarray*}
where the series in the right hand side converges uniformly. According to equalities
\begin{eqnarray*}
\int_{[0,1]^d}{S^{(d)}_{2^k}}\Bigl(\frac{\partial f}{\partial z},x\Bigr)\,d\mu(x)\\
=\sum_{(n_1,\dots,n_d)\in\mathbb{N}_k} c^{(d)}_{n_1, \dots, n_d}\Bigl(\frac{\partial f}{\partial z}\Bigr)\int_{[0,1]^d} \chi^{(d)}_{n_1, \dots, n_d}(x)\,d\mu(x)\\
\stackrel{\eqref{eq:estcn}}{=}\sum_{(n_1,\dots,n_d)\in\mathbb{N}_k} \frac{\partial c^{(d)}_{n_1, \dots, n_d}(f)}{\partial z}\int_{[0,1]^d} \chi^{(d)}_{n_1, \dots, n_d}(x)\,d\mu(x)\\
=\frac{d}{d z}\int_{[0,1]^d}{S^{(d)}_{2^k}}(f,x)\,d\mu(x)
\end{eqnarray*}
and Remark \ref{rm1}, we can differentiate the series in \eqref{eq:defetas} and obtain that
\begin{eqnarray*}
\dfrac{d\tilde{\eta}(z)}{dz}=\frac{d}{dz}\int_{[0,1]^d}S^{(d)}_{1}(f, x)\,d\mu(x)\\
+\sum_{k=1}^{\infty}\frac{d}{dz}\int_{[0,1]^d}\bigl(S^{(d)}_{2^k}(f, x)-S^{(d)}_{2^{k-1}}(f, x)\bigr)\,d\mu(x)
=\int_{[0,1]^d}\dfrac{\partial f(z, x)}{\partial z}\,d\mu(x),
\end{eqnarray*}
which finishes the proof.
\end{proof}

\section{Integral as a function of upper limit}\label{sc:upper}

For continuous function $f(x): [0,1]^d\to\rr$ and  $y=(y_1,\dots,y_d)\in [0,1]^d$, an we consider the random function
\begin{equation}\label{eq:defxi}
\xi(y)=\int_{\prod_{s=1}^d [0,y_s]}f(x)\,d\mu(x).
\end{equation}

Lemma~\ref{lm:haarun} implies that
\begin{equation}\label{eq:defverxi}
\begin{split}
\tilde{\xi}(y)=\int_{\prod_{s=1}^d [0,y_s]}S^{(d)}_{1}(f, x)\,d\mu(x)\\
+\sum_{k=1}^{\infty}\int_{\prod_{s=1}^d [0,y_s]}(S^{(d)}_{2^k}(f, x)-S^{(d)}_{2^{k-1}}(f, x))\,d\mu(x).
\end{split}
\end{equation}
is the version of $\xi(y)$.

We refer to the following assumptions on SM $\mu$.

\begin{assumption}\label{assbc} Random function $\mu(x)=\mu\Bigl(\prod_{i=1}^{d}[0,x_i]\Bigr),\ x\in[0,1]^d$,
has continuous paths.
\end{assumption}

In the following condition on the uniform modulus of continuity we take the continuous version of $\mu$.

\begin{assumption}\label{assbmc}
If $d\geq 2$ then $\sum_{k=1}^\infty k^{d-2}\omega(\mu, 2^{-k})<\infty\quad \textrm{a.\ s.}$
\end{assumption}

It is easy to see that A\ref{assbmc} holds for $\mu(x)$ with H\"{o}lder continuous paths. It also holds if, for example, $\omega(\mu,\tau)\leq C|\ln\tau|^{-\epsilon}$, $\epsilon>d-1$.
\begin{example}
Denote
\[
q(\tau)=|\ln \tau|^{-\gamma},\,\gamma>d-1/2,\qquad \mathcal{K}(z)=\sqrt{\dfrac{d q^2(z)}{d z}}.
\]
We introduce the following stochastic process:
\begin{equation*}
B^{q}(x)=\int_{\prod_{s=1}^d [0,x_s]} \prod_{s=1}^d \mathcal{K}(x_s-y_s) dW(y),\ x=(x_1,\dots,x_d)
\end{equation*}
where $W$ is a $d$-dimensional Wiener process. Theorem 3.1 in~\cite{hino23} implies the existence of rectangle $[t,T]=\prod_{i=1}^d[t_i,T_i]\subset [0,1]^d$ such that
\begin{equation}
\label{eqgauss2}
\lim_{\varepsilon\to 0+}\sup_{\substack{x,\bar{x}\in[t,T]\\ \delta_{x,\bar{x}}\leq \varepsilon}}\frac{|B^q(x)-B^q(\bar{x})|}{\delta_{x,\bar{x}}\sqrt{\ln\Bigl(\frac{D}{q^-1(\delta_{x,\bar{x}})}\Bigr)}}=C\ \text{a. s.}
\end{equation}
Here $\delta_{x,\bar{x}}=\|B^q(x)-B^q(\bar{x})\|_{L^2(\Omega)}\leq Cq(|x-\bar{x}|)$, $D$ is a diameter of $[t,T]$. From \eqref{eqgauss2} it follows that for all $x,\bar{x}\in[t,T]$
\begin{eqnarray*}
|B^q(x)-B^q(\bar{x})|\leq C\delta_{x,\bar{x}}\sqrt{\ln D+\delta_{x,\bar{x}}^{-1/\gamma}}\\
\stackrel{\gamma>1/2}{\leq}C|\ln|x-\bar{x}||^{-\gamma}\sqrt{\ln D+|\ln|x-\bar{x}||}\\
\leq C|\ln|x-\bar{x}||^{1/2-\gamma}.
\end{eqnarray*}
Therefore, $\omega_{[t,T]}(B^q,\tau)\leq C|\ln\tau|^{1/2-\gamma}$.
\end{example}
Now we are ready to formulate the main result of the section.
\begin{thm}\label{th:conpx}
Let Assumptions A\ref{assbc} and A\ref{assbmc} hold, and the function $f(x): [0,1]^d\to\rr$ be continuously differentiable $d$ times on $[0,1]^d$.

Then, for the random function $\xi$ defined by \eqref{eq:defxi}, version~\eqref{eq:defverxi} has continuous paths on $[0,1]^d$ a.~s.
\end{thm}

\begin{proof}
For version~\eqref{eq:defverxi}, we have that
\begin{equation}\label{eq:txilim}
\tilde{\xi}(y)=\lim_{k\to\infty} \int_{\prod_{s=1}^d [0,y_s]}S^{(d)}_{2^k}(f,x)\,d\mu(x).
\end{equation}
Here $S^{(d)}_{2^k}(f,x)$ is a simple function. By our assumption, SM $\mu$ has a continuous paths.
Therefore, for each $k$, the random function of variable $y$
\[
\int_{\prod_{s=1}^d [0,y_s]}S^{(d)}_{2^k}(f,x)\,d\mu(x)
\]
has a continuous paths. We will prove that the convergence in~\eqref{eq:txilim} is uniform in $y\in [0,1]^d$ a.s., and this will imply the continuity of $\tilde{\xi}$.

Our aim is to show that
\begin{equation*}
\sum_{k=1}^{\infty}\sup_{y\in [0,1]^d}\Bigl|\int_{\prod_{s=1}^d [0,y_s]}(S^{(d)}_{2^k}(f, x)-S^{(d)}_{2^{k-1}}(f, x))\,d\mu(x)\Bigr|\le \tilde{C}^{(d)}_{f,\mu}(\omega)
\end{equation*}
for some random constant $\tilde{C}^{(d)}_{f,\mu}(\omega)<\infty$ a. s. that may depends on $d,f,\mu$.

Recall that paths of $\mu(x_1,\dots,x_d)$ are continuous, therefore for any cut of the set $\prod_{s=1}^d[y_{s1},y_{s2}]\subset [0,1]^d$ we have
\begin{equation}\label{eq:muzero}
\mu\Bigl(\prod_{s=1}^d[y_{s1},y_{s2}]\cap \{x_s=a\}\Bigr)=0.
\end{equation}

Thus, we obtain
\begin{equation*}
\begin{split}
\int_{\prod_{s=1}^d [0,y_s]}(S^{(d)}_{2^k}(f, x)-S^{(d)}_{2^{k-1}}(f, x))\,d\mu(x)\\
=\sum_{(n_1,\dots,n_d)\in \nn_k\setminus \nn_{k-1}} c^{(d)}_{n_1, \dots, n_d}(f)\int_{\prod_{s=1}^d [0,y_s]} \chi^{(d)}_{n_1, \dots, n_d}(x)\,d\mu(x)\\
\stackrel{\eqref{eq:muzero}}{=}   \sum_{(n_1,\dots,n_d)\in \nn_k\setminus \nn_{k-1}} c^{(d)}_{n_1, \dots, n_d}(f)2^{(j_1+\dots+j_d)/2}\\
\times \int_{\prod_{s=1}^d [0,y_s]}
\prod_{1\le s\le d,n_s\ge 2}\Bigl(\ii_{\Delta_{n_s}^+}(x_s)-\ii_{\Delta_{n_s}^-}(x_s)\Bigr)\,d\mu(x)\\
= \sum_{(n_1,\dots,n_d)\in \nn_k\setminus \nn_{k-1}} c^{(d)}_{n_1, \dots, n_d}(f) 2^{(j_1+\dots+j_d)/2}\\
\times \sum_{\varepsilon_s\in\{+, -\}} \mu\Bigl(\prod_{1\le s\le d}(\Delta_{n_s})^{\varepsilon_s}\cap \prod_{1\le s\le d} [0,y_s]\Bigr) \\
:=A_{k1}(y)+A_{k2}(y).
\end{split}
\end{equation*}

Here $A_{k1}(y)$ is the sum of terms with $(n_1,\dots,n_d)\in \nn_k\setminus \nn_{k-1}$, $\varepsilon_s\in\{+, -\}$ such that
\[
\prod_{1\le s\le d}(\Delta_{n_s})^{\varepsilon_s}\subset \prod_{s=1}^d [0,y_s]
\]
(here for $n_s=1$ we take only $\varepsilon_s=+$ and $(\Delta_{n_s})^{\varepsilon_s}=(0,1)$).
Taking into account the definition of $(\Delta_{n_s})^{+}$ and $(\Delta_{n_s})^{-}$ in \eqref{eq:defdel}, we get that in $A_{k1}(y)$
\[
\prod_{1\le s\le d}(\Delta_{n_s})^{\varepsilon_s}\cap \prod_{s=1}^d [0,y_s]=\prod_{1\le s\le d}(\Delta_{n_s})^{\varepsilon_s}=\prod_{1\le s\le d}\Delta_{m_s}
\]
for some $(m_1,\dots,m_d)\in \nn_{k+1}\setminus \nn_{k}$. Therefore,
\begin{equation*}
\begin{split}
\sum_{k=1}^\infty \sup_{y\in[0,1]^d} |A_{k1}(y)|\\
\le \sum_{k=1}^\infty  \sum_{(n_1,\dots,n_d)\in \nn_k\setminus \nn_{k-1}} |c^{(d)}_{n_1, \dots, n_d}(f)| 2^{(j_1+\dots+j_d)/2}\sum_{\varepsilon_s\in\{+, -\}}\Bigl|\mu\Bigl(\prod_{1\le s\le d}(\Delta_{n_s})^{\varepsilon_s}\Bigr)\Bigr|\\
\stackrel{\eqref{estc}}{\le}  C\sum_{k=1}^\infty  \sum_{(n_1,\dots,n_d)\in \nn_k\setminus \nn_{k-1}} 2^{-(j_1+\dots+j_d)}\sum_{\varepsilon_s\in\{+, -\}}\Bigl|\mu\Bigl(\prod_{1\le s\le d}(\Delta_{n_s})^{\varepsilon_s}\Bigr)\Bigr|\\
\le C\sum_{k=1}^\infty  \sum_{(m_1,\dots,m_d)\in \nn_{k+1}\setminus \nn_{k}} 2^d 2^{-(j_1'+\dots+j_d')}\Bigl|\mu\Bigl(\prod_{1\le s\le d}\Delta_{m_s}\Bigr)\Bigr|.
\end{split}
\end{equation*}
Here $j'_s$ are taken such that $2^{j'_s}+1\le m_s\le 2^{j'_s+1}$, and $j'_s=j_s+1$ for respective $m_s$.
From estimates in \eqref{eqP1}, \eqref{eqP2} it follows that
\[
\sum_{k=1}^\infty \sup_{y\in[0,1]^d} |A_{k1}(y)|\le \tilde{C}^{(d)}_{\mu}(\omega).
\]
for some random constant $\tilde{C}^{(d)}_{\mu}(\omega)$ that depends only of $d$ and $\mu$. (It is easy to see that all estimates in  \eqref{eqP1}, \eqref{eqP2} remain valid if we change $\chi^{(d)}_{n_1, \dots, n_d}(x)$ to $2^{(j_1+\dots+j_d)/2}\prod_{1\le s\le d}\ii_{\Delta_{m_s}}(x_s)$.)

Further, we estimate $A_{k2}(y)$, i. e. the sum of terms with $(n_1,\dots, n_d)\in\nn_k\setminus \nn_{k-1}$ such that $y_s\in \Delta_{n_s}$
for some $s$.

We get
\begin{eqnarray*}
|A_{k2}(y)|\le \sum_{\substack{(n_1,\dots,n_d)\in \mathbb{N}_k\setminus \mathbb{N}_{k-1},\\ \exists y_s\in (\Delta_{n_s})^{\varepsilon_s}}}
2^{-\sum_{1\le s\le d,n_s\ge 2}j_s} \Bigl|\mu\Bigl(\prod_{1\le s\le d}((\Delta_{n_s})^{\varepsilon_s}\cap  [0,y_s])\Bigr)\Bigr|\\
=2^d\sum_{\substack{(m_1,\dots,m_d)\in \mathbb{N}_k\setminus \mathbb{N}_{k-1},\\ \exists y_s\in \Delta_{m_s}}}
2^{-\sum_{1\le s\le d,m_s\ge 3}j'_s} \Bigl|\mu\Bigl(\prod_{1\le s\le d}(\Delta_{m_s}\cap  [0,y_s])\Bigr)\Bigr|\\
\leq 2^d\omega(\mu, 2^{-k-1})\sum_{\substack{(m_1,\dots,m_d)\in \mathbb{N}_{k+1}\setminus \mathbb{N}_{k},\\ \exists y_s\in \Delta_{m_s}}}
2^{-\sum_{1\le s\le d,m_s\ge 3}j'_s}:=D_{k+1}.
\end{eqnarray*}

Now we check the convergence of the series $\sum_{k=1}^\infty D_k$ with the help of the sum
\[
\sum_{\substack{(m_1,\dots,m_d)\in \mathbb{N}_k,\\ \exists y_s\in \Delta_{m_s}}}
2^{-\sum_{1\le s\le d,m_s\ge 3}j'_s}:=T_k.
\]
Notice that for each fixed set $(j'_1,\dots,j'_d)$ there exist at most
\[
2^{\sum_{1\le s\le d,m_s\ge 3} j'_s}-\prod_{1\le s\le d,m_s\ge 3}\bigl(2^{j'_s}-1\bigr)
\]
sets $(m_1\dots,m_d)$, for which $\exists y_s\in \Delta_{m_s}$. Denoting the set of indexes $s$, which satisfy the equality $n_s=1$, as $A$, we obtain that
\begin{eqnarray*}
\sup_{y}T_k\leq\sum_{A\subset\{1,\dots,d\}}\sum_{\substack{1\leq j_i'\leq k,\\i\in A}}\biggl(1-\prod_{i\in A}\bigl(1-2^{-j'_i}\bigr)\biggr)\\
=\sum_{A\subset\{1,\dots,d\}}\sum_{\substack{1\leq j_i'\leq k,\\i\in A}}\sum_{\substack{B\subset A,\\B\neq\emptyset}}(-1)^{|B|+1}2^{-\sum_{i\in B}j'_i}\\
\leq \sum_{A\subset\{1,\dots,d\}}\sum_{\substack{B\subset A,\\B\neq\emptyset}}(-1)^{|B|+1}\sum_{\substack{1\leq j_i'\leq k,\\i\in A}}2^{-\sum_{i\in B}j'_i}\\
\stackrel{(*)}{=}\sum_{u=1}^{d}\sum_{v=1}^{u}(-1)^{v+1}\left(\hspace{-3pt}\begin{array}{c}d\\u\end{array}\hspace{-3pt}\right)\left(\hspace{-3pt}\begin{array}{c}u\\v\end{array}\hspace{-3pt}\right)\sum_{\substack{1\leq j'_i\leq k,\\ 1\leq i\leq u}}2^{-\sum_{1\leq i\leq v}j'_i}\\
=\sum_{u=1}^d\sum_{v=1}^u (-1)^{v+1}\left(\hspace{-3pt}\begin{array}{c}d\\u\end{array}\hspace{-3pt}\right)\left(\hspace{-3pt}\begin{array}{c}u\\
v\end{array}\hspace{-3pt}\right)(1-2^{-k})^vk^{u-v}\\
=\sum_{u=1}^d\left(\hspace{-3pt}\begin{array}{c}d\\u\end{array}\hspace{-3pt}\right)(k^{u}-(k-1+2^{-k})^u)\\
=(k+1)^d-(k+2^{-k})^d.
\end{eqnarray*}
Here in $(*)$ we used the fact that sum $(-1)^{|B|+1}\sum_{\substack{1\leq j_i'\leq k,\\i\in A}}2^{-\sum_{i\in B}j'_i}$ depends only on $|A|$ and $|B|$. Therefore,
\begin{eqnarray*}
D_k\leq 2^d\omega(\mu, 2^{-k})\bigl((k+1)^d-(k+2^{-k})^d-k^d+(k-1+2^{-k+1})\bigr)\\
=2^d\omega(\mu, 2^{-k})\bigl((k+1)^d-2k^d+(k-1)^d+O(2^{-\beta k})\bigr),
\end{eqnarray*}
where $0<\beta<1$. Applying Assumption~A\ref{assbmc} we get the statement of the theorem.
\end{proof}

If $\mu(x)$ have a H\"{o}lder continuous paths with exponent $\gamma>0$, then the statement of Theorem~\ref{th:conpx} follows from Theorem~16~\cite{harang21}. Moreover, Theorem~16~\cite{harang21} states that $\xi$ has the version that is H\"{o}lder continuous with the same exponent $\gamma>0$.
In \cite{harang21}, integral is considered in the Young sense, and its value coincides with value of our integral with respect to H\"{o}lder continuous $\mu$.

\begin{thm}
\label{th:bsintd}
Let Assumptions A\ref{assbm}, A\ref{assbc} and A\ref{assbmc} hold, and the function $f(x): [0,1]^d\to\rr$ be continuously differentiable $d$ times on $[0,1]^d$.

Then, for the random function $\xi$ defined by \eqref{eq:defxi}, for any $1\le p<+\infty,\ 0<\alpha<\min\{1/p,1/2\}$ version~\eqref{eq:defverxi} with probability 1 belongs to the Besov space $B^\alpha_{p,p}([0,1]^d)$.
\end{thm}

\begin{proof}
The statement follows from Theorems \ref{thnbspd} and \ref{th:conpx}.
\end{proof}

\bibliographystyle{plain}
\bibliography{ManikinRadchenkoIntRmBib}

\end{document}